\documentclass{article}
\usepackage{amsthm,amsmath}
\usepackage{amscd,amssymb,latexsym,graphicx}
\usepackage{amscd}
\usepackage[all]{xy}
\usepackage{cite}
\usepackage[
                     ]{hyperref}
\title{An almost existence theorem for non-contractible
         periodic orbits in cotangent bundles}
\author{
   Pedro A.S. Salom\~{a}o
      \\ IME USP
   \and
   Joa Weber
      \\ UNICAMP
   }
\date{18 April 2013}
\newtheorem{theoremABC}{Theorem}

\newtheorem{theorem}{Theorem}[section]
\newtheorem{corollary}[theorem]{Corollary}

\newtheorem{lemma}[theorem]{Lemma}
\newtheorem{proposition}[theorem]{Proposition}

\theoremstyle{definition}
\newtheorem{definition}[theorem]{Definition}

\newtheorem{example}[theorem]{Example}
\theoremstyle{remark}
\newcommand{\N}{{\mathbb{N}}}

\newcommand{\R}{{\mathbb{R}}}

\newcommand{\Z}{{\mathbb{Z}}}
\newcommand{\Hh}{{\mathcal{H}}}

\newcommand{\Ll}{{\mathcal{L}}}   % Lagrangian planes
\newcommand{\Pp}{{\mathcal{P}}}

\newcommand{\cHZ}{{\rm c_{HZ}}}       % Hofer-Zehnder capacity
\newcommand{\cBPS}{{\rm c_{BPS}}}       % BPS capacity
\newcommand{\eps}{{\varepsilon}}

\def\NABLA#1{{\mathop{\nabla\kern-.5ex\lower1ex\hbox{$#1$}}}}
\def\Nabla#1{\nabla\kern-.5ex{}_{#1}}
\def\Tabla#1{\Tilde\nabla\kern-.5ex{}_{#1}}
\def\abs#1{\mathopen|#1\mathclose|}
\def\Abs#1{\left|#1\right|}

\def\Norm#1{\left\|#1\right\|}

\renewcommand{\Tilde}{\widetilde}

\begin{document}

\maketitle

%%%%%%%%%%%%%%%%%%%%%%%%%%%%%%%%%%%%%%%
%%%%%%%%%%%%%%%% Abstract %%%%%%%%%%%%%%%%
%%%%%%%%%%%%%%%%%%%%%%%%%%%%%%%%%%%%%%%
\begin{abstract}
Assume $M$ is a closed connected smooth manifold
and $H:T^*M\to\R$ a smooth proper function bounded
from below. Suppose the sublevel set $\{H<d\}$
contains the zero section $M$
and $\alpha$ is a non-trivial homotopy class
of free loops in $M$. Then for almost every $s\in[d,\infty)$
the level set $\{H=s\}$ carries a periodic orbit
$z$ of the Hamiltonian system $(T^*M,\omega_0,H)$
representing $\alpha$.
Examples show that the condition $\{H<d\}\supset M$
is necessary and almost existence cannot be improved
to everywhere existence.
\end{abstract}

%\keywords{
                 {\small 2010
                 {\it Mathematics Subject Classification.}
                 37-06 (Primary),
         % 37-XX Dynamical systems and ergodic theory
         % 37-06 Proceedings
                 70H12 (Secondary)
         % 70-XX Mechanics of particles and systems
         % 70Hxx Mechanics of particles and systems
         % 70H12 Periodic and almost periodic solutions
                 }
%}

%\keywords{{\small {\bf Key words:}
%                 Hamiltonian dynamics }}

%\tableofcontents

%%%%%%%%%%%%%%%%%%%%%%%%%%%%%%%%%%%%%%%
%%%%%%%%%%%%%%%%%%%%%%%%%%%%%%%%%%%%%%%
%%%%%%%%%%%%%%%% Section %%%%%%%%%%%%%%%%%
%%%%%%%%%%%%%%%%%%%%%%%%%%%%%%%%%%%%%%%
%%%%%%%%%%%%%%%%%%%%%%%%%%%%%%%%%%%%%%%
\section{Introduction and main result}

Suppose $M$ is a smooth manifold and its cotangent
bundle $\pi:T^*M\to M$ is equipped with the canonical
symplectic structure $\omega_0=-d\theta$.
Here $\theta$ ($=pdq$) denotes the canonical Liouville
$1$-form on $T^*M$.
We view the elements of $T^*M$
as pairs $(q,p)$ where $q\in M$ and $p\in T_q^*M$.
Given any function $H$ on $T^*M$,
the identity $dH=\omega_0(X_H,\cdot)$
uniquely determines the Hamiltonian vector
field $X_H$ on $T^*M$. The integral curves
of $X_H$ are called (Hamiltonian) orbits. They preserve the
level sets of the total energy $H$. Of particular
interest are periodic orbits, namely orbits
$\gamma:\R\to T^*M$ such that
$\gamma(t+T)=\gamma(t)$ for some constant $T>0$
and all $t\in\R$. The
infimum\footnote{
    Here and thoughout we use the convention
    $\inf\emptyset=\infty$.
  }
over such $T$ is called the period of $\gamma$.
Given a family of energy levels, the question
arises which levels carry a periodic orbit.

Existence of a periodic orbit on a dense set of energy
levels was proved for $T^*\R^n$ by Hofer and
Zehnder~\cite{1987-HZ-InventMath} in 1987 and
for $T^*M$ by Hofer and Viterbo~\cite{1988-HV-ASNS}
in 1988.
The result for $T^*\R^n$ was extended
to existence almost everywhere by
Struwe~\cite{1990-Struwe-BullBMS} in 1990.
Existence of \emph{non-contractible} periodic orbits was
studied among others in 1997 by
Cieliebak~\cite{1997-Cieliebak-destroy} on
starshaped levels in $T^*M$, in 2000 by Gatien and
Lalonde~\cite{2000-GL-Duke} employing
Lagrangian submanifolds, and in 2003 by
Biran, Polterovich, and Salamon~\cite{2003-BPS-Duke}
on $T^*M$ for $M=\R^n/\Z^n$
or $M$ closed and negatively curved.
The dense existence theorem in~\cite{2003-BPS-Duke}
was generalized in 2006 to all closed Riemannian manifolds
in~\cite{2006-Joa-Duke}.
Theorem~\ref{thm:almex} below
%The main result of the present paper
is the corresponding almost existence theorem.
In contrast the almost existence theorem of Macarini and
Schlenk~\cite{2005-MacSchl-BLMS} requires finiteness
of the $\pi_1$-sensitive Hofer-Zehnder capacity.
An assumption that has been verified
to the best of our knowledge only for such
cotangent bundles which carry certain circle actions;
see~\cite{2004-Macarini-CCM ,
2011arXiv-Irie-noncontractible}.
For further references concerning dense
and almost existence results we refer
to~\cite{2005-Ginzburg-PM} and
concerning non-contractible orbits
to~\cite{2012arXiv-Gurel-noncontractible}.

\begin{theoremABC}[Almost existence]\label{thm:almex}
Assume $M$ is a closed connected smooth manifold and
$H:T^*M\to\R$ is a
proper\footnote{A map is called proper if
preimages of compact sets are compact.}
smooth function bounded from below. Suppose the
sublevel set
$\{H<d\}$ contains $M$. Then for every non-trivial
homotopy class $\alpha$ of free loops in $M$
the following is true.
For almost every $s\in[d,\infty)$ the level set
$\{H=s\}$ carries a periodic Hamiltonian
orbit $z$ that represents $\alpha$ in the sense that
$[\pi\circ z]=\alpha$ where $\pi:T^*M\to M$
is the projection map.
\end{theoremABC}

\begin{proof}
There are three main ingredients in the proof.
The main player is the
Biran-Polterovich-Salamon (BPS)~\cite{2003-BPS-Duke}
capacity $\cBPS$ whose monotonicity axiom
Proposition~\ref{prop:monotonicity}
naturally leads to the monotone function
$c_\alpha:[d,\infty)\to[0,\infty]$
%defined by
\begin{equation}\label {eq:c_alpha}
     c_\alpha(s):=\cBPS(\{H<s\},M;\alpha).
\end{equation}

Secondly, the existence result~\cite[Thm.~A]{2006-Joa-Duke}
concerning periodic orbits enters as follows:
A priori the range of $\cBPS$ includes $\infty$
(by Definition~\ref{def:BPS} this is the case if no $1$-periodic
orbit representing $\alpha$ exists).
To prove finiteness of the function
$c_\alpha$ pick as an auxiliary quantity
a Riemannian metric $g$ on $M$.
Then using~\cite[Thm.~A]{2006-Joa-Duke} one readily
calculates that the BPS capacity of the open unit
disk cotangent
bundle relative to its zero section is equal to the smallest
length $\ell_\alpha$ among all closed geodesics representing
$\alpha$; see~\cite[Thm.~4.3]{2006-Joa-Duke}.
The rescaling argument in Lemma~\ref{le:rescaling}
shows that the capacity of the open radius $r$ disk
cotangent bundle $D_rT^*M$ is $r\ell_\alpha$.
Observe that $\{H\le s\}$ is compact since
$H$ is proper and bounded below.
Hence the set $\{H<s\}$ is bounded and therefore
contained in $D_rT^*M$ for some sufficiently
large radius $r=r(s)$. Thus $c_\alpha(s)\le r(s)\ell_\alpha$ by
the monotonicity axiom and this proves finiteness
of $c_\alpha$.

Thirdly, by Lebesgue's last theorem,
see e.g.~\cite{2002-Pugh-Analysis},
it is well known, yet amazing, that monotonicity
of the map $c_\alpha:[d,\infty)\to[0,\infty)$
implies differentiability, thus Lipschitz continuity,
at almost every point $s$ in the sense of measure theory.
Now the key input is Theorem~\ref{thm:carrries-orbit}
whose proof is by an analogue of the Hofer-Zehnder
method~\cite[Sec.~4.2]{1994-HZ-book} and which
detects for each such $s$
a periodic orbit on the corresponding level
set $\{H=s\}$.
\end{proof}

%%%%%%%%%%%%%%%% Subsection  %%%%%%%%%%%%%%
%\subsubsection*{Necessary condition}
\begin{example}[Necessary condition]
The condition
$\{H<d\}\supset M$ cannot be dropped
in Theorem~\ref{thm:almex}.
First of all, together with $H$ being proper
and bounded below, it guarantees that
each level set $\{H=s\}$ is actually nonempty whenever
$s\in[d,\infty)$.
Now consider a pendulum. It moves on $M=S^1=\R/\Z$
in a potential of the form $V(q)=1+\cos 2\pi q$;
see Figure~\ref{fig:fig-bead-potential}.
The Hamiltonian
$H:T^*M=S^1\times\R\to\R$
is given by $H(q,p)=\tfrac12 p^2+V(q)$; see
Figure~\ref{fig:fig-bead-phase} for the phase portrait.
\begin{figure}%[h]
\begin{minipage}[b]{.49\linewidth}
  \centering
  \includegraphics%[width=0.75\textwidth]
                             {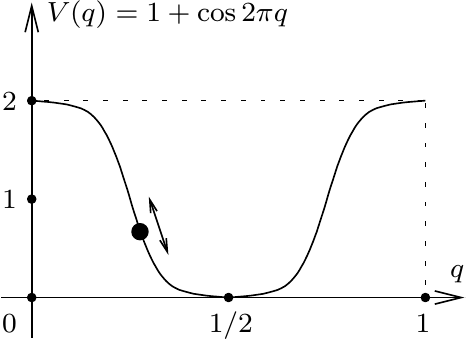}
  \caption{Potential energy $V$}
  \label{fig:fig-bead-potential}
\end{minipage}
\hfill
\begin{minipage}[b]{.49\linewidth}
  \centering
  \includegraphics%[width=0.75\textwidth]
                             {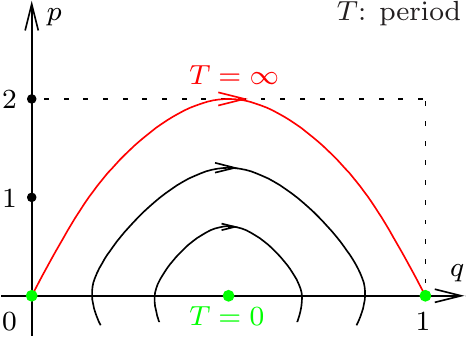}
  \caption{Pendulum phase portrait}
  \label{fig:fig-bead-phase}
\end{minipage}
\hfill
\end{figure}
     Energies below the
     maximum value $2$ of the potential $V$ do not allow
     for full rotations.
     For such low energies the pendulum can just
     swing hence and forth.
Observe that $\{H<1\}\not\supset M$.
On the other hand, for any energy
$s\in[1,2)$ the level set $\{H=s\}$
consists of a periodic orbit which is contractible
onto the stable (lower) equilibrium point
$(x,y)\equiv(\tfrac12,0)$. So none of these orbits
represents a homotopy  class $\alpha\not=0$.
(For $s>2$ the sets $\{H=s\}$ represent
classes $\alpha\not=0$.
The set $\{H=2\}$ consists of
the unstable (upper) equilibrium point and two homoclinic
orbits one of them indicated red in
Figure~\ref{fig:fig-bead-phase}.)
\end{example}

%%%%%%%%%%%%%%%% Subsection  %%%%%%%%%%%%%%
%\subsubsection*{Existence everywhere does not hold}
\begin{example}[Existence everywhere not true]
To see that \emph{almost} existence in
Theorem~\ref{thm:almex}
cannot be improved to \emph{everywhere} existence
consider the case $M=S^1$ and a
Hamiltonian $H:S^1\times \R\to\R$ of the form
$H(q,p)=h(p)$.
More precisely, pick a proper smooth function $h\ge 0$ with
$h(0)=0$ and $h(\pm 3)=2$ and where the points
$0,\pm 3$ are the only points of slope zero;
see Figure~\ref{fig:fig-Hamiltonian-h}.
\begin{figure}[b]
  \centering
  \includegraphics{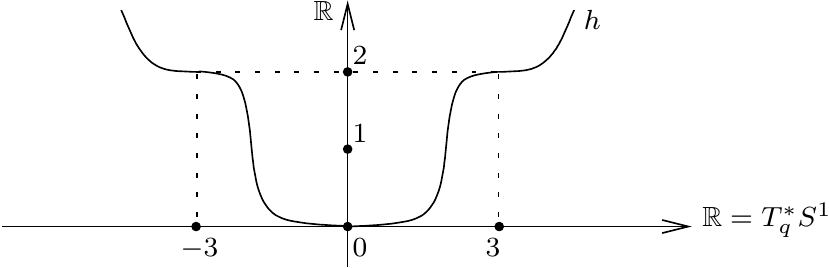}
  \caption{Hamiltonian $H=h(p)$ without
                 non-constant orbits on $\{H=2\}$
                 }
  \label{fig:fig-Hamiltonian-h}
\end{figure}
Then $\{H<1\}$ contains $M=S^1$.
Moreover, the whole level set $\{H=2\}$
consists of critical points of $H$. Therefore
on $\{H=2\}$ the Hamiltonian vector field $X_H$
vanishes identically and so all orbits are necessarily constant.

In contrast to this critical level counterexample
it should be interesting to find a regular
level of a smooth Hamiltonian $H$
as in Theorem~\ref{thm:almex}
without a periodic orbit in a given
homotopy class $\alpha\not=0$.
One possible way to achieve this is to start with an
energy level with finitely many periodic orbits
representing $\alpha$, then destroy them
using the symplectic plugs constructed
in~\cite{1997-Ginzburg-IMRN-CE}.

For general symplectic manifolds
existence may fail completely;
see~\cite{1987-Zehnder-remarks}
and~\cite{2011arXiv-Usher-infiniteHZ} for examples
of closed symplectic manifolds  admitting Hamiltonians
with no non-constant periodic orbits.
\end{example}

%%%%%%%%%%%%%%%%%%%%%%%%%%%%%%%%%%%%%%%
%%%%%%%%%%%%%%%%%%%%%%%%%%%%%%%%%%%%%%%
%%%%%%%%%%%%%%%% Section %%%%%%%%%%%%%%%%%
%%%%%%%%%%%%%%%%%%%%%%%%%%%%%%%%%%%%%%%
%%%%%%%%%%%%%%%%%%%%%%%%%%%%%%%%%%%%%%%
\section{Symplectic capacities} 

To fix notation
consider $\R^{2n}$ with coordinates
$(x_1,\ldots,x_n,y_1,\ldots,y_n)$ and symplectic form
$\omega_0 =\sum_{i=1}^n dx_i \wedge dy_i.$
Associate to each symplectic manifold $(N,\omega)$, of fixed dimension
$2n>0$ and possibly with boundary, a number $c(N,\omega) \in [0,\infty]$  that satisfies the axioms:
\begin{itemize}
\item {\bf Monotonicity:}
$c(N_1,\omega_1) \leq c(N_2,\omega)$ whenever there is
a symplectic embedding $\psi: (N_1,\omega_1) \to (N_2,\omega_2)$.
\item {\bf Conformality:} $c(N,\lambda \omega) =
\abs{\lambda}\, c(N,\omega), \forall \lambda \in \R \setminus \{0\}$.
\item {\bf Non-Triviality:} $c(B(1),\omega_0) = c(Z(1),\omega_0)=\pi$.
\end{itemize}
Here $B(r) = \{(x,y) \in \R^{2n}: |x|^2 +|y|^2< r^2\}$ is the
ball of radius $r>0$ and $Z(r)= \{(x,y) \in \R^{2n}: x_1^2 + y_1^2<r^2\}$ is the symplectic cylinder of radius $r>0$. On $(\R^{2n},\omega_0)$, one
checks  the following re-scaling property:
$$\text{$U\subset \R^{2n}$  open}\quad \Rightarrow\quad c(\lambda U, \omega_0) = \lambda^2 c(U,\omega_0),\forall \lambda \in \R \setminus\{0\}.$$
A map $(N,\omega) \mapsto c(N,\omega)$ satisfying the three axioms above is called a symplectic capacity.
Gromov introduced this notion in~\cite{1985-Gromov-Inventiones} and showed that $$c_0(N,\omega) := \sup\left\{\pi r^2: \exists \mbox{ symplectic embedding }\psi:(B(r),\omega_0) \to (N,\omega)\right\}$$ is a symplectic capacity, called Gromov's width. It satisfies $c_0(N,\omega) \leq c(N,\omega)$ for any other symplectic capacity $c$. One of the consequences of the existence of a symplectic capacity is the non-squeezing theorem which asserts that $$\exists \mbox { symplectic embedding }\psi: (B(r), \omega_0) \to (Z(R),\omega_0)\quad\Leftrightarrow\quad r\leq R.$$

%%%%%%%%%%%%%%%%%%%%%%%%%%%%%%%%%%%%%%%
%%%%%%%%%%%%%%%% Subsection  %%%%%%%%%%%%%%
%%%%%%%%%%%%%%%%%%%%%%%%%%%%%%%%%%%%%%%
\subsection{Hofer-Zehnder capacity}

Hofer and Zehnder introduced in \cite{1994-HZ-book} a symplectic
capacity defined in terms of the Hamiltonian dynamics on the
underlying symplectic manifold $(N,\omega)$. Recall that a smooth
function $H:N \to \R$ determines the Hamiltonian vector field $X_H$ by
$i_{X_H} \omega = dH$. We say that a periodic orbit of $\dot x = X_H
(x)$ is fast if its period is $<1$. A function $H:N\to\R$ is
called admissible if it admits a maximum and the following conditions hold:

\begin{itemize}
\item $0 \leq H\leq \max H<\infty$.
\item $\exists K \subset N\setminus \partial N$ compact, such that $H|_{N\setminus K} = \max H$.
\item $\exists U \subset N$ open and non-empty, such that $H|_U =0$.
\item $\dot x = X_H \circ x$ admits no non-constant fast periodic orbits.
\end{itemize}
The set of admissible Hamiltonians is denoted by
$\mathcal{H}_a(N,\omega)$. Let 
$$\cHZ(N,\omega) := \sup \{\max H \mid H \in \mathcal{H}_a(N,\omega)\}.$$

\begin{theorem}[Hofer-Zehnder]\label{theo_hz}$\cHZ$ is a symplectic capacity. \end{theorem}

We should remark that the hard part of proving Theorem \ref{theo_hz} is to show that $\cHZ$ satisfies the non-triviality axiom.

%%%%%%%%%%%%%%%%%%%%%%%%%%%%%%%%%%%%%%%
%%%%%%%%%%%%%%%% Subsection  %%%%%%%%%%%%%%
%%%%%%%%%%%%%%%%%%%%%%%%%%%%%%%%%%%%%%%
\subsection{BPS relative capacity}

Fix a closed manifold $M$.
The components $\Ll_\alpha M$
of the free loop space
$\Ll M:=C^\infty(S^1,M)$
are labelled by the elements
$\alpha=[\gamma]$ of the set $\tilde\pi_1(M)$
of homotopy classes of free loops $\gamma$ in $M$.
Here and throughout we identify $S^1$ with $\R/\Z$
and think of $\gamma$ as a smooth map
$\gamma:\R\to M$ that satisfies
$\gamma(t+1)=\gamma(t)$ for every $t\in\R$.
A function $H\in C^\infty_0(S^1\times T^*M)$
determines a $1$-periodic family of compactly supported
vector fields $X_{H_t}$ on $T^*M$ by
$dH_t=\omega_0(X_{H_t},\cdot)$. Let
$$
     \Pp_1(H;\alpha)
     :=\{z:S^1\to T^*M\mid
     \dot z(t)=X_{H_t}(z(t))\,\forall t\in S^1,[\pi\circ z]=\alpha
     \}
$$
be the set of $1$-periodic orbits of $X_{H_t}$ whose
projections to $M$ represent $\alpha$.

\begin{definition}\label{def:BPS}
Following~\cite{2003-BPS-Duke} assume
$W\subset T^*M$ is an open subset which contains
the zero section $M$.
For any constant $b>0$ consider the set
$$
     \Hh_b(W)
     :=\left\{ H\in C^\infty_0(S^1\times W)\,\Big|\,
     m_0(H):=\max_{S^1\times M} H\le -b
     \right\}.
$$
The BPS capacity of $W$ relative $M$ and with respect
to $\alpha\in\tilde\pi_1(M)$ is given by
\begin{equation}\label{eq:BPS-capacity}
     \cBPS(W,M;\alpha)
     :=\inf\left\{b>0\mid
     \text{$\Pp_1(H;\alpha)\not=\emptyset$
     for every $H\in\Hh_b(W)$}
     \right\}.
\end{equation}
\end{definition}
Note that $\cBPS$ takes values in $[0,\infty]$
since we use the convention $\inf\emptyset=\infty$.
Furthermore, the BPS capacity 
%is not a symplectic capacity, it should rather be called
is a relative symplectic capacity.

\begin{proposition}
[Monotonicity{\cite[Prop.~3.3.1]{2003-BPS-Duke}}]
\label{prop:monotonicity}
If $W_1\subset W_2\subset T^*M$ are open subsets
containing $M$
and $\alpha\in\tilde \pi_1(M)$, then
$
     \cBPS(W_1,M;\alpha)\le \cBPS(W_2,M;\alpha)
$.
\end{proposition}

%%%%%%%%%%%%%%%% Subsubsection %%%%%%%%%%%%
%\subsubsection*{Rescaling}

Fix a Riemannian metric on $M$
and constants $r,b>0$.
Let $DT^*M$ be the open unit disk cotangent
bundle and $D_rT^*M$ the one of radius $r$. Observe that
\begin{equation}\label{eq:Hh_c}
     H\in\Hh_b(DT^*M)
     \quad\Longleftrightarrow\quad
     H_r\in\Hh_{rb}(D_rT^*M)
\end{equation}
whenever the Hamiltonians $H$ and $H_r$ are related by
$
     H_r(t,q,p)=r\cdot H(t,q,\frac{p}{r})
$.
In addition, pick $\alpha\in\tilde\pi_1(M)$.
Then there is the crucial bijection
\begin{equation}\label{eq:Pp_T}
\begin{split}
     \Pp_1(H;\alpha)
   \to \Pp_1(H_r;\alpha)
   :
     (x,y)
   \mapsto (x,ry)
\end{split}
\end{equation}
asserting that the $1$-periodic orbits of $H$
correspond naturally with those of $H_r$.

\begin{lemma}[Rescaling]\label{le:rescaling}
$
     \cBPS(D_rT^*M,M;\alpha)
     =r\cdot \cBPS(DT^*M,M;\alpha)
$.
\end{lemma}

\begin{proof}
By definition~(\ref{eq:BPS-capacity})
of the BPS capacity
%and by~(\ref{eq:Hh_c}) and~(\ref{eq:Pp_T})
we obtain that
\begin{equation*}
\begin{split}
     &\cBPS(D_rT^*M,M;\alpha)
   \\
     &=\inf\left\{rb>0\mid
     \text{$\Pp_1(H_r;\alpha)\not=\emptyset$
     for every $H_r\in\Hh_{rb}(D_rT^*M)$}
     \right\}
   \\
    &=r\cdot \inf\left\{b>0\mid
     \text{$\Pp_1(H;\alpha)\not=\emptyset$
     for every $H\in\Hh_b(DT^*M)$}
     \right\}
   \\
    &=r\cdot \cBPS(DT^*M,M;\alpha)
\end{split}
\end{equation*}
where the second step
uses~(\ref{eq:Hh_c}) and~(\ref{eq:Pp_T}).
\end{proof}

\begin{corollary}\label{cor:BPS-radius-r}
$
     \cBPS(D_rT^*M,M;\alpha)
     =r\ell_\alpha
$
where $\ell_\alpha$ is the smallest length among all
closed curves representing $\alpha$.
\end{corollary}

\begin{proof}
$
     \cBPS(DT^*M,M;\alpha)
     =\ell_\alpha
$
by~\cite[Thm.~4.3]{2006-Joa-Duke}.
Apply Lemma~\ref{le:rescaling}.
\end{proof}

%%%%%%%%%%%%%%%%%%%%%%%%%%%%%%%%%%%%%%%
%%%%%%%%%%%%%%%%%%%%%%%%%%%%%%%%%%%%%%%
%%%%%%%%%%%%%%%% Section %%%%%%%%%%%%%%%%%
%%%%%%%%%%%%%%%%%%%%%%%%%%%%%%%%%%%%%%%
%%%%%%%%%%%%%%%%%%%%%%%%%%%%%%%%%%%%%%%
\section{The Hofer-Zehnder method}

Assume the Hamiltonian $H:T^*M\to\R$ is smooth,
proper, and bounded from below and a sublevel set
$\{H<d\}$ contains $M$. Fix a non-trivial homotopy
class $\alpha$ of free loops in $M$.
Consider the monotone function $c_\alpha$ defined
on the interval $[d,\infty)$ by~(\ref{eq:c_alpha}).
By Lebesgue's last theorem, see
e.g.~\cite[p.~401]{2002-Pugh-Analysis},
the function $c_\alpha$ is differentiable at almost every
point in the sense of measure theory.

\begin{theorem}\label{thm:carrries-orbit}
Assume $s_0\in[d,\infty)$ is a regular value of $H$ and
$c_\alpha$ is Lipschitz continuous at $s_0$.
Then the hypersurface $S=H^{-1}(s_0)$ carries a
periodic orbit $z_T$ of $X_H$ that represents
$\alpha$ and where $T>0$ indicates the period.
\end{theorem}

\begin{proof}
The proof is an adaption of the Hofer-Zehnder
method~\cite[Sec.~4.2]{1994-HZ-book}
to the case at hand. To emphasize this we
mainly keep their notation.
Fix $s_0$ as in the hypothesis of the theorem.
Then $S_0:=H^{-1}(s_0)$ is a
hypersurface\footnote{
     A hypersurface is a
     smooth submanifold of codimension $1$.
  }
in $T^*M$ by the inverse function theorem. It is
compact since $H$ is proper and it bounds the open set
$\dot B_0:=\{H<s_0\}$ since $H$ is bounded below.
Furthermore, by the implicit
function theorem and compactness of $S_0$ there is
a constant $\mu>0$ such that $s_0+\eps$ is a
regular value of $H$ and $S_\eps:=H^{-1}(s_0+\eps)$ is
diffeomorphic to $S_0$ whenever $\eps\in[-\mu,\mu]$.
Note that $S_\eps$ bounds the open set
$\dot B_\eps:=\{H<s_0+\eps\}$ which itself
contains the zero section $M$ of $T^*M$.
Furthermore, since $c_\alpha$ is Lipschitz continuous
at $s_0$ there is a constant $L>0$ such that
\begin{equation}\label{eq:c-L}
     c(\eps)-c(0)\le L\eps
     ,\qquad
     c(\eps):=c_\alpha(s_0+\eps),
\end{equation}
for every $\eps\in[-\mu,\mu]$; otherwise,
choose $\mu>0$ smaller. We proceed in 3 steps.

I. Pick $\tau\in(0,\mu)$. Then there is a Hamiltonian
$K\in C^\infty_0(S^1\times\dot B_0)$ whose
maximum over the zero section satisfies
$$
     -c(0)
     <m_0(K)
     \le -\left(c(0)-L\tau\right)
$$
and which does not admit any $1$-periodic orbit
representing $\alpha$. Indeed if no such $K$ exists, then
$\cBPS(\dot B_0,M;\alpha)\le c(0)-L\tau$
by definition~(\ref{eq:BPS-capacity}) of the BPS capacity.
But $c(0)=c_\alpha(s_0)=\cBPS(\dot B_0,M;\alpha)$
and we obtain the contradiction
$c(0)\le c(0)-L\tau$.
Now pick a smooth function
$f:\R\to [-3L\tau,0]$ such that
\begin{alignat*}{2}
     &f(s)=-3L\tau&\qquad\text{if}&\qquad s\le 0
   \\
    &f(s)=0&\qquad\text{if}&\qquad s\ge \tfrac{\tau}{2}
   \\
    &0<f^\prime(s)\le 7L&\qquad\text{if}
    &\qquad 0<s<\tfrac{\tau}{2}
\end{alignat*}
and consider the Hamiltonian
$F\in C^\infty_0(S^1\times \dot B_\tau)$ defined by
\begin{alignat*}{2}
    &F(x)=K(x)-3L\tau&\qquad\text{if}
    &\qquad x\in \dot B_0
   \\
    &F(x)=f(\eps)&\qquad\text{if}
    &\qquad x\in S_\eps=H^{-1}(s_0+\eps),\; 0\le\eps<\tau
   \\
    &F(x)=0&\qquad\text{if}
    &\qquad x\notin \dot B_\tau
\end{alignat*}
and illustrated by Figure~\ref{fig:fig-Hamiltonian-F}.
\begin{figure}%[h]
  \centering
  \includegraphics{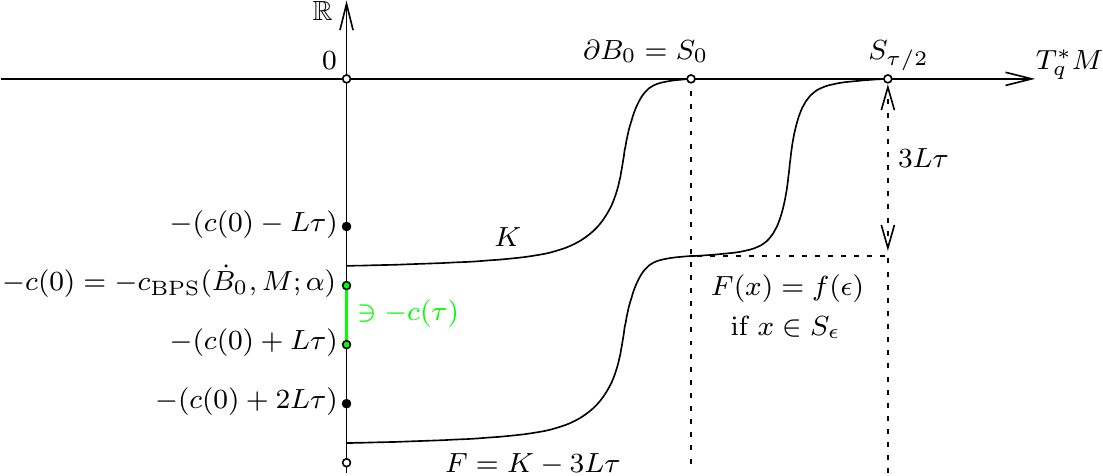}
  \caption{Hamiltonians $F\in\Hh_{c(\tau)}(\dot B_\tau)$
                 and $K$ with $\Pp_1(K;\alpha)=\emptyset$
                 }
  \label{fig:fig-Hamiltonian-F}
\end{figure}
By~(\ref{eq:c-L}) the Hamiltonian $F$
satisfies the estimate
$$
     m_0(F)
     =m_0(K)-3L\tau
     \le-(c(0)-L\tau)-3L\tau
     <-(c(0)+L\tau)
     \le-c(\tau).
$$
Since $m_0(F)\le-\cBPS(\dot B_\tau,M;\alpha)$
the definition~(\ref{eq:BPS-capacity}) of the
BPS capacity shows that the set $\Pp_1(F;\alpha)$
is not empty. In other words, there is a $1$-periodic
orbit $z$ of $X_F$ that represents $\alpha$.
Observe that $z$ cannot intersect $\dot B_0$: Due
to compact support the open set $\dot B_0$ is invariant
under the flow of $K$. But the flows of $K$ and
$K-3L\tau=F$ coincide. Thus, if $z$ intersects
$\dot B_0$, then it stays completely inside. But this
is impossible since $\Pp_1(K;\alpha)=\emptyset$.
\\
On the other hand, since $\alpha\not=0$ the orbit
$z$ of $X_F$ is non-constant and therefore it must
intersect the regions foliated by the hypersurfaces
$S_\eps$ where $0<\eps<\frac{\tau}{2}$.
But each of them is a level set of $F$, hence
invariant under the flow of $X_F$.
This shows that $z$ lies on $S_\eps$ for some
$0<\eps<\frac{\tau}{2}$.

II. Repeat the argument for each element of a sequence
$\tau_j\to 0$ to obtain sequences $F_j$ and $\eps_j$
and a sequence $z_j$ of $1$-periodic orbits of
$X_{F_j}$ that lie on $S_{\eps_j}$ and where $\eps_j\to 0$.
Next we interpret each $z_j$ as a $T_j$-periodic
orbit of $X_H$ by rescaling time. Most importantly, the
periods $T_j$ are uniformly bounded from above by $7L$.
To see this note that on the open set
$$
     U:=\bigcup_{\eps\in(-\mu,\mu)} S_\eps
     =\bigcup_{\eps\in(-\mu,\mu)} H^{-1}(s_0+\eps)
$$
the Hamiltonian $H$ is obviously given by
$H(x)=s_0+\eps$ whenever $x\in S_\eps$.
For each $\tau_j$ and each $\eps\in[0,\tau_j)$ we have
$$
     F_j(x)=f_j(H(x)-s_0)=f_j(\eps)
$$
for every $x\in S_\eps$. At such $x$ use the definition
of $X_{F_j}$ and the chain rule to get
$$
     \omega_0\left(X_{F_j},\cdot\right)
     =dF_j
     =f_j^\prime(H-s_0) dH
     =\omega_0\left(f_j^\prime(\eps) X_H,\cdot\right).
$$
Thus, because $z_j$ lies on $S_{\eps_j}$, it satisfies
the equation
$$
     \dot z_j(t)=X_{F_j}\circ z_j(t)
     =T_j\cdot X_H \circ z_j(t)
     ,\qquad
     T_j:=f_j^\prime(\eps_j),
$$
and the periodic boundary condition
$z_j(t+1)=z_j(t)$ for every $t\in\R$.

III. Uniform boundedness of the periods
$T_j$ is crucial in the following proof of existence of a
$1$-periodic orbit $z$ of $X_H$ which lies on the
original level hypersurface $S_0=H^{-1}(s_0)$ and
represents the given class $\alpha$. Indeed note that
$S_{\eps_j}\subset\{H\le s_0+\mu\}=:B_\mu$
and that $B_\mu$ is compact since $H$ is proper and
bounded below. In other words, the sequence of loops
$z_j$ is uniformly bounded in $C^0$.
Concerning $C^1$ we obtain the uniform estimate
$$
     \Abs{\dot z_j(t)}
     =\Abs{T_j}\cdot\Abs{X_H\circ z_j(t)}
     \le 7L \Norm{X_H}_{C^0(B_\mu)}
$$
for all $t\in S^1$ and $j\in\N$.
Therefore by the Arzel\`{a}-Ascoli theorem
there is a subsequence, still denoted by $z_j$,
which converges in $C^0$ and by using the equation
for $z_j$ even in $C^\infty$ to
a smooth $1$-periodic solution $z$ of
the equation $\dot z=T\cdot X_H(z)$
where $T=\lim_{j\to\infty} T_j$.
Since $\eps_j\to 0$ the orbit $z$
takes values on the desired level hypersurface
$S_0=H^{-1}(s_0)$.
To prove that $z=(x,y)$ represents the same class $\alpha$
as does each $z_j=(x_j,y_j)$ we need to show that
$[x]=[x_j]$
%observe that there are the obvious
%homotopies $(q,p)\sim (q,0)$ and $(q_j,0)\sim (q_j,p_j)$.
%Thus it remains to show that $x\sim x_j$
for some $j$.
To see this consider the injectivity radius $\iota>0$
of the compact Riemannian manifold $(M,g)$ and
pick $j$ sufficiently large such that the Riemannian
distance between $x(t)$ and $x_j(t)$ is less than $\iota/2$
for every $t\in S^1$. Setting $\exp_{x(t)}\xi(t)=x_j(t)$
provides the desired homotopy $h_\lambda(t)=\exp_{x(t)}\lambda\xi(t)$
between $h_0=x$ and $h_1=x_j$.
\\
Reparametrize time to obtain the $T$-periodic
solution $z_T(t):=z(t/T)$ of
$$
     \dot z_T(t)=\tfrac{1}{T} \dot z(t/T)
     =X_H\circ z(t/T)=X_H\circ z_T(t)
$$
which obviously represents the same class $\alpha$
as $z$. Since $\alpha\not=0$ the loop $z_T$ cannot be
constant and so the period necessarily satisfies $T>0$.
This concludes the proof of
Theorem~\ref{thm:carrries-orbit}.
\end{proof}

%\newpage
%\vspace{.2cm}
\noindent
{\sc First author.}
{\small
         Financial support:
FAPESP grant 2011/16265-8 and CNPq grant 303651/2010-5.
         Author address:
Rua do Mat\~ao 1010, Cidade Universit\'aria, S\~ao Paulo, SP, Brazil.
         Email:
psalomao@gmail.com
}

\vspace{.2cm}
\noindent
{\sc Second author.}
{\small
         Financial support:
            FAPESP grant 2011/01830-1.
         Author address:
            IMECC,
%                 UNICAMP,
%                 Instituto de Matem\'{a}tica, Estat\'{\i}stica
%                 e Computa\c{c}\~{a}o Scient\'{\i}fica,
%                 Universidade Estadual de Campinas
            Rua S\'{e}rgio Buarque de Holanda 651,
%                 Cidade Universit\'{a}ria "Zeferino Vaz",
%                 Distr. Bar\~{a}o Geraldo,
            13083-859 Campinas, SP, Brazil.
            Email: joa@math.sunysb.edu
}

%%%%%%%%%%%%%%%%%%%%%%%%%%%%%%%%%%%%%
%%%%%%%%%%%%%%%% References  %%%%%%%%%%%%
%%%%%%%%%%%%%%%%%%%%%%%%%%%%%%%%%%%%%

\bibliography{almex v2.bbl}{}
%\bibliography{BibDesk-Joa}{}
\bibliographystyle{plain}

\end{document}